\documentclass[12pt,reqno]{amsart}

\usepackage[utf8]{inputenc}
\usepackage[T1]{fontenc}
\usepackage{times}
\usepackage{mathrsfs}
\usepackage[dvips]{graphics}
\usepackage{epsfig}
\usepackage{amsmath,amsthm,amssymb,amscd}
\usepackage{mathtools}
\usepackage{color}
\usepackage{hyperref}
\usepackage[margin=1in]{geometry}
\usepackage{enumitem}
\allowdisplaybreaks

\newcommand{\burl}[1]{\textcolor{blue}{\url{#1}}}

\newcommand{\nnend}{\nonumber\\}

\newcommand\be{\begin{equation}}
\newcommand\ee{\end{equation}}
\newcommand\bea{\begin{eqnarray}}
\newcommand\eea{\end{eqnarray}}
\newcommand\bi{\begin{itemize}}
\newcommand\ei{\end{itemize}}
\newcommand\ben{\begin{enumerate}}
\newcommand\een{\end{enumerate}}

\theoremstyle{plain}
\newtheorem{thm}{Theorem}[section]
\newtheorem{conj}[thm]{Conjecture}
\newtheorem{cor}[thm]{Corollary}
\newtheorem{lem}[thm]{Lemma}

\theoremstyle{definition}
\newtheorem{exa}[thm]{Example}
\newtheorem{defi}[thm]{Definition}

\newtheorem{nota}[thm]{Notation}

\theoremstyle{remark}
\newtheorem{rek}[thm]{Remark}

\numberwithin{equation}{section}

\renewcommand{\[}{\begin{equation}}
\renewcommand{\]}{\end{equation}}

\author[E.~Bo\l dyriew]{El\.zbieta~Bo\l dyriew}
\email{\textcolor{blue}{\href{mailto:eboldyriew@colgate.edu}{eboldyriew@colgate.edu}}}
\address{Department of Mathematics, Colgate University, Hamilton, NY 13346}

\author[J.~Haviland]{John~Haviland}
\email{\textcolor{blue}{\href{mailto:havijw@umich.edu}{havijw@umich.edu}}}
\address{Department of Mathematics, University of Michigan, Ann Arbor, MI 48109}

\author[P.~L\^am]{Ph\'uc~L\^am}
\email{\textcolor{blue}{\href{mailto:plam6@u.rochester.edu}{plam6@u.rochester.edu}}}
\address{Department of Mathematics, University of Rochester, Rochester, NY 14627}

\author[J.~Lentfer]{John~Lentfer}
\email{\textcolor{blue}{\href{mailto:jlentfer@hmc.edu}{jlentfer@hmc.edu}}}
\address{Department of Mathematics, Harvey Mudd College, Claremont, CA 91711}

\author[S.~J.~Miller]{Steven~J.~Miller}
\email{\textcolor{blue}{\href{mailto:sjm1@williams.edu}{sjm1@williams.edu}},  \textcolor{blue}{\href{Steven.Miller.MC.96@aya.yale.edu}{Steven.Miller.MC.96@aya.yale.edu}}}
\address{Department of Mathematics and Statistics, Williams College, Williamstown, MA 01267}

\author[F.~Trejos~Su\'arez]{Fernando~Trejos~Su\'arez}
\email{\textcolor{blue}{\href{mailto:fernando.trejos@yale.edu}{fernando.trejos@yale.edu}}}
\address{Department of Mathematics, Yale University, New Haven, CT 06520}

\title{An Introduction to Completeness of Positive Linear Recurrence Sequences}

\date{\today}

\thanks{This research was conducted as part of the SMALL 2020 REU at Williams College. The authors were supported by NSF Grants DMS1947438 and DMS1561945, Williams College, Yale University, and the University of Rochester.}


\begin{document}
	
\begin{abstract}
	 A positive linear recurrence sequence (PLRS) is a sequence defined by a homogeneous linear recurrence relation with positive coefficients and a particular set of initial conditions. A sequence of positive integers is \emph{complete} if every positive integer is a sum of distinct terms of the sequence. One consequence of Zeckendorf's theorem is that the sequence of Fibonacci numbers is complete. Previous work has established a generalized Zeckendorf's theorem for all PLRS's. We consider PLRS's and want to classify them as complete or not. We study how completeness is affected by modifying the recurrence coefficients of a PLRS. Then, we determine in many cases which sequences generated by coefficients of the forms $[1, \ldots, 1, 0, \ldots, 0, N]$ are complete. Further, we conjecture bounds for other maximal last coefficients in complete sequences in other families of PLRS's. Our primary method is applying Brown's criterion, which says that an increasing sequence $\{H_n\}_{n = 1}^{\infty}$ is complete if and only if $H_1 = 1$ and $H_{n + 1} \leq 1 + \sum_{i = 1}^n H_i$. This paper is an introduction to the topic that is explored further in \cite{BHLLMT}.
	
\end{abstract}

\maketitle

\tableofcontents

\section{Introduction}

Edouard Zeckendorf famously proved that every positive integer can be written uniquely as a sum of non-consecutive Fibonacci numbers, when indexed $\{1,2,3,5,\dots\}$; this unique decomposition is called the \emph{Zeckendorf decomposition} \cite{Ze}. The property of unique decompositions has been generalized to a much larger class of linear recurrence relations, called PLRS's. The following definitions are from \cite{MW, BBGILMT}.

\begin{defi}\label{defn:goodrecurrencereldef} We say a sequence $\{H_n\}_{n=1}^\infty$ of positive integers is a \textbf{Positive Linear Recurrence Sequence (PLRS)} if the following properties hold:

\ben
\item \emph{Recurrence relation:} There are non-negative integers $L, c_1, \dots, c_L$\label{c_i} such that \be H_{n+1} \ = \ c_1 H_n + \cdots + c_L H_{n+1-L},\ee  with $L, c_1$ and $c_L$ positive.
\item \emph{Initial conditions:} $H_1 = 1$, and for $1 \le n < L$ we have
\be H_{n+1} \ =\
c_1 H_n + c_2 H_{n-1} + \cdots + c_n H_{1}+1.\ee
\een
\end{defi}

\begin{defi}[Legal decompositions]
We call a decomposition $\sum_{i=1}^{m} {a_i H_{m+1-i}}$\label{a_i} of a positive integer $N$ (and the sequence $\{a_i\}_{i=1}^{m}$) legal \label{legal} if $a_1>0$, the other $a_i \ge 0$, and one of the following two conditions holds:
\ben

\item We have $m<L$ and $a_i=c_i$ for $1\le i\le m$.

\item There exists $s\in\{1,\dots, L\}$ such that
\begin{equation}
a_1\ = \ c_1,\ a_2\ = \ c_2,\ \cdots,\ a_{s-1}\ = \ c_{s-1}\ {\rm{and}}\ a_s<c_s,
\end{equation}
$a_{s+1}, \dots, a_{s+\ell} \ = \  0$ for some $\ell \ge 0$,
and $\{b_i\}_{i=1}^{m-s-\ell}$ (with $b_i = a_{s+\ell+i}$) is legal or empty.

\een
\end{defi}

The following theorem is due to \cite{GT}, and stated in this form in \cite{MW}.
\begin{thm}[Generalized Zeckendorf's Theorem for PLRS]\label{thm:genZeckendorf} Let $\{H_n\}_{n=1}^\infty$ be a \emph{Positive Linear Recurrence Sequence}. Then there is a unique legal decomposition for each positive integer $N\ge 0$.
\end{thm}

The goal of this paper is to provide an introduction to the \emph{completeness} of PLRS's. This definition is from \cite{Br, HK}.

\begin{defi}
An arbitrary sequence of positive integers $\{f_i\}_{i=1}^\infty$ is \textbf{complete} if and only if every positive integer $n$ can be represented in the form $n=\sum_{i=1}^\infty \alpha_i f_i$, where $\alpha_i \in \{0,1\}$. A sequence that fails to be complete is \textbf{incomplete}.
\end{defi}

In other words, a sequence of positive integers is complete if and only if each positive integer can be written as a sum of unique terms of the sequence. The Fibonacci numbers are a motivating example. 

\begin{exa}
    The Fibonacci sequence, indexed from $\{1,2,\ldots\}$ is complete. This sequence, in particular with the correct initial conditions, is the PLRS defined by $H_{n+1}=H_n+H_{n-1}$. Then completeness follows from Zeckendorf's Theorem, as every positive integer has a unique decomposition, and critically, no sequence terms are used more than once. In fact, Zeckendorf's Theorem is a stronger statement than what is required for completeness.  Completeness does not require the decompositions to be unique, nor that they use only nonconsecutive terms.
\end{exa}

After seeing this example, does Theorem~\ref{thm:genZeckendorf} imply that all PLRS's are complete? Previous work in numeration systems by Gewurz and Merola \cite{GM} has shown that specific classes of recurrences as defined by Fraenkel \cite{Fr} are complete under their greedy expression. However, we cannot generalize this result to all PLRS's. For legal decompositions, the decomposition rule might permit sequence terms to be used more than once. This is not allowed for completeness decompositions, where each unique term from the sequence can be used at most once.
    \begin{exa}
        The PLRS $H_{n+1} = H_n + 3H_{n-1}$ has terms $\{1, 2, 5, 11, \ldots\}$. The unique \emph{legal} decomposition for $9$ is $1\cdot 5 + 2\cdot 2$, where the term $2$ is used twice. However, no \emph{complete} decomposition for $9$ exists. Adding all terms from the sequence less than $9$ is $1 + 2 + 5 = 8$, and to include $11$ or any subsequent term surpasses $9$.
    \end{exa}
It is not realistic to check that all terms of an infinite sequence have decompositions that use each term no more than once. Instead, we make use of the following criterion for completeness of a sequence, due to \cite{Br}. It allows us to simplify proving completeness for many specific PLRS's to induction proofs. 
\begin{thm}[Brown's Criterion]
	If $a_n$ is a nondecreasing sequence, then $a_n$ is complete if and only if $a_1 = 1$ and for all $n > 1$,
	\begin{equation}\label{eqn:BrownsCrit}
		a_{n + 1} \leq 1 + \sum_{i = 1}^{n} a_i.
	\end{equation}
\end{thm}

\begin{nota}
We use the notation $[c_1, \ldots, c_L]$, which is the collection of all $L$ coefficients, to represent the PLRS $H_{n+1} = c_1 H_n + \cdots + c_L H_{n+1-L}$.
\end{nota}

A simple case to consider is when all coefficients in $[c_1,\ldots,c_L]$ are strictly positive. The following result, proved in Section \ref{sec:modifying}, completely characterizes these sequences are either complete or incomplete. 

\begin{thm}\label{basic}
    If $\{H_n\}$ is a PLRS generated by all positive coefficients $[c_1,\ldots,c_L]$, then sequence is complete if and only if the coefficients are $[\underbrace{1,\ldots,1}_L]$ or $[\underbrace{1,\ldots,1}_{L-1},2]$ for $L \geq 1$.
\end{thm}

The situation becomes much more complicated when we consider all PLRS's that have at least one $0$ as a coefficient. In order to be able to make progress on determining completeness of these PLRS's, we develop several additional tools. The following three theorems are results that allow certain modifications of the coefficients $[c_1, \ldots, c_L]$ that generate a PLRS that is known to be complete or incomplete, and preserve completeness or incompleteness. They are proven in Section \ref{sec:modifying}.

\begin{thm}\label{lem:incompAddCoeff}
    Consider sequences $\{ G_{n} \} = [c_1,\dots, c_{L}]$ and $\{ H_{n} \}= [c_1,,\dots, c_{L},c_{L+1}]$, where $c_{L+1}$ is any positive integer. If $\{ G_{n} \}$ is incomplete, then $\{ H_{n} \}$ is incomplete as well.
\end{thm}
\begin{thm}\label{decreaseLastCoe}
    Consider sequences $\{G_n\}=[c_1,\ldots, c_{L-1}, c_L]$ and $\{H_n\}=[c_1,\ldots, c_{L-1}, k_L]$, where $1 \leq k_L \leq c_L$. If $\{G_n\}$ is complete, then $\{H_n\}$ is also complete.
    \end{thm}
\begin{thm}\label{Adding M Theorem}    
    Consider sequences $\{G_n\}=[c_1,\ldots, c_{L-1}, c_L]$ and $\{H_n\}=[c_1,\ldots, c_{L-1} + c_L]$. If $\{G_n\}$ is incomplete, then $\{H_n\}$ is also incomplete.
\end{thm}

The next two theorems are results that classify two families of PLRS's as complete or incomplete. They are shown in Section \ref{sec:famiilies}.

\begin{thm}\label{thm:1onekzero} The sequence generated by $[1,\underbrace{0,\ldots,0}_k,N]$ is complete if and only if $1 \leq N \leq \left\lceil(k+2)(k+3)/{4}\right\rceil$.
\end{thm}

\begin{thm}\label{2onekzero}
The sequence generated by $[1, 1, \underbrace{0, \dots, 0}_{k}, N]$ is complete if and only if $1 \leq N \leq \lfloor (F_{k + 6} - k - 5)/4 \rfloor$, where $F_n$ are the Fibonacci numbers with $F_1=1, F_2=2$.
\end{thm}

We have a partial extension of these theorems to when there are $g$ initial ones followed by $k$ zeroes in the collection of coefficients. For a proof, see \cite{BHLLMT}.

\begin{thm}\label{thm:gbon}
Consider a PLRS generated by coefficients $[\underbrace{1, \dots, 1}_{g}, \underbrace{0,\ldots,0}_{k},N]$, with $g,k \geq 1$.
\ben[itemsep=2ex, leftmargin=2em]
    \item For $g \geq k +\lceil \log _2 k\rceil$, the sequence is complete if and only if $1 \leq N \leq 2^{k+1}-1$.
    \item For $k \leq g \leq k +\lceil \log _2 k\rceil$, the sequence is complete if and only if $1 \leq N \leq 2^{k+1} - \lceil k/{2^{g-k}} \rceil$.
    \een
\end{thm}

This paper is an introduction to the classification of PLRS's by completeness and serves as an introduction to the full results, including an analysis of the principal root of the recurrence relation's characteristic function, in \cite{BHLLMT}.

\section{Modifying Sequences}\label{sec:modifying}

A basic question to ask is how far we can tweak the coefficients used to generate a sequence, yet preserve its completeness. The modifying process turns out to be well-behaved and heavily dependent on the location of coefficients that are changed. Before we start looking into implementing any changes to our sequences, we first need to understand the maximal complete sequence.

\subsection{The Maximal Complete Sequence}
The maximal complete sequence is the sequence that has terms that grow as quickly as possible while the sequence remains complete. For example, if a sequence begins $\{1, t,\ldots \}$, what can $t$ possibly be for the sequence to be complete? The sequence is increasing as a result of the specific initial conditions we are using, until the full recurrence relation takes over. So except in the degenerate case of $H_{n+1} = H_{n}$, i.e., the coefficient collection is just $[c_1=1]$, the sequence is strictly increasing. On the other hand, if $t \geq 3$, then there is no way to create a decomposition for $2$ that uses sequence terms only once. This means that the maximal complete sequence has $t=2$. Extending this idea, we establish the following lemma. 

\begin{lem}\label{clm:largestCompleteGaps}
The complete increasing sequence with maximal terms is $\{a_n\} = \{2^{n - 1}\}$.
\end{lem}
\begin{proof}
It is straightforward to see $\{a_n\} = \{ 2^{n-1} \}$ is generated by the PLRS $H_{n+1} = 2H_n$. This is complete by Brown's criterion, since for any $n$, \[
2^n = 1+\sum_{i=1}^{n}2^{i-1}.
\]
Observe that by using a strict equality here with Brown's criterion, we are ``maximizing'' the complete sequence.

Now, let $\{ b_n\}$ be an increasing sequence of positive integers, and suppose for some $n$, $b_n > 2^{n-1}$, i.e., at some index $n$, the sequence $\{b_n\}$ exceed that of the sequence $\{2^{n-1}\}$. Note that there are precisely $2^{n-1}-1$ non-empty subsets of $\{ b_1,\ldots, b_{n-1}\}$, and thus at most $2^{n-1}-1$ positive integers which can be expressed as a sum of these values. Thus, as the set $\{1, 2, \ldots, b_n -1 \}$ has at least $2^{n-1}$ elements, at least one of those elements cannot be written as a sum of integers in $\{ b_1,\ldots, b_{n-1}\}$, and so the sequence is not complete. Hence, we conclude that $\{2^{n-1}\}$ is the maximal complete sequence. 
\end{proof}
Now we can look at all complete sequences with only positive coefficients. 
\begin{proof}[Proof of Theorem \ref{basic}.]
Assuming completeness of the sequence, by the definition of a PLRS and by Brown's criterion, we have
\begin{equation}
c_1 H_{L - 1} + c_2 H_{L - 2} + \cdots + c_{L - 1} H_{1} + 1 = H_L \leq 1 + H_1 + H_2 + \cdots + H_{L - 1}.
\end{equation}
Since $c_i \geq 1$ for $1 \leq i \leq L$, $c_i = 1$ for $1 \leq i < L$.
By the definition of a PLRS,
\begin{equation}
H_{L+1} = c_1 H_L + c_2 H_{L - 1} + \cdots + c_L H_1 = H_L + H_{L - 1} + \cdots + H_2 + c_L H_1.
\end{equation}
Which together with Brown's criterion gives $c_L H_1 \leq 1 + H_1 = 2.$
And so $c_L\leq 2$, which completes the forward direction of the proof.

Conversely, we know that  the sequence $[2]$ is complete by Lemma~\ref{clm:largestCompleteGaps}.  Thus, let us assume that $c_1 = \cdots = c_{L - 1} = 1$ and $1 \leq c_L \leq 2$. We prove that $H_n$ satisfies Brown's criterion. We can show this explicitly for $1 \leq n < L$ and by strong induction on $n$ further on,
where the inductive hypothesis is applied to $H_{n + 1 - L}$ to obtain \begin{equation}
    H_{n + 2}\leq H_{n + 1} + \cdots + H_{n + 2 - L} + H_{n + 1 - L} + (H_{n - L} + \cdots + H_1 + 1),
\end{equation}
which completes the proof.
\end{proof}

 A specific case of Theorem \ref{basic} is that a PLRS with coefficients $[\underbrace{1,\ldots,1}_{L-1},2]$ is complete. A consequence of Lemma~\ref{clm:largestCompleteGaps} it that $\{H_k\} = \{2^{k - 1}\}$ is an inclusive upper bound for any complete sequence. A careful reader might note that these two results are related. Due to a PLRS's specific initial conditions, we can prove that this sequence $\{2^{k - 1}\}$ can be generated by multiple collections of coefficients. The proof, by strong induction, can be found in \cite{BHLLMT}.

\begin{cor}
A PLRS with coefficients $[\underbrace{1,\ldots,1}_{L-1},2]$ generates the sequence $H_n= 2^{n-1}$.
\end{cor}


\subsection{Modifications of Sequences of Arbitrary Coefficients}
Modifying coefficients in order to preserve completeness proves to be a balancing act. Sometimes increasing a coefficient causes an incomplete sequence to become complete, while other times, increasing a coefficient causes a complete sequence to become incomplete. For example, $[1,0,0,0,0,0,15]$ is incomplete; increasing the second coefficient to $1$, i.e., $[1,1,0,0,0,0,15]$ is complete. Further increasing it to $2$, i.e., $[1,2,0,0,0,0,15]$ is again incomplete. To study how such modifications preserve completeness or incompleteness, we add a new definition to our toolbox. 

\begin{defi}
    For a sequence $\{H_n\}$, we define its \textbf{$\boldsymbol{n}$th Brown's gap}
    \begin{equation}
    B_{H, n} := 1 + \sum_{i=1}^{n-1}H_i - H_n.
    \end{equation}
\end{defi}
Thus, from Brown's criterion, $\{H_n\}$ is complete if and only if $B_{H, n} \ge 0$ for all $n \in \mathbb{N}$.\\



So, what happens if we append one more coefficient to $[c_1,\ldots,c_L]$? It turns out that if our sequence is already incomplete, appending any new coefficients will never make it complete. This is Theorem \ref{lem:incompAddCoeff}, which using are ready to prove using Brown's gap. 

\begin{proof}[Proof of Theorem \ref{lem:incompAddCoeff}.]
By Brown's criterion, it is clear that $\{ G_{n} \}$ is incomplete if and only if there exists $n$ such that $B_{G,n}<0$. We claim that for all $m$, $B_{H,m}\leq B_{G,m}$. If true, our lemma is proven: suppose $B_{G,n}<0$ for some $n$, we would see $B_{H,n}\leq B_{G,n}<0$, implying $\{ H_{n} \}$ is incomplete as well.

We proceed by induction. Clearly, $B_{H,k}=B_{G,k}$ for $1\leq k \le L$. Further, for $k=L$, we see \[
	B_{G,L+1}-B_{H,L+1}= 1+\sum_{i=1}^{L}G_{i} - G_{L+1} - \left(1+\sum_{i=1}^{L}H_{i} - H_{L+1} \right) =H_{L+1}-G_{L+1}=1>0
.\] 
Now, let $m \ge 2$ be arbitrary, and suppose 
\begin{equation}\label{Bequ1}
B_{H,\; L+m-1}\leq B_{G,\; L+m-1}.
\end{equation}
We wish to show that $B_{H,\; L+m}\leq B_{G,\; L+m}$.  Note that 
\begin{equation}\label{1first}
B_{H,\; L+m}-B_{H,\; L+m-1}= 2H_{L+m-1} - H_{L+m}.
\end{equation}
Similarly, 
\begin{equation}\label{2first}
B_{G,\; L+m}-B_{G,\; L+m-1}= 2G_{L+m-1} - G_{L+m}.
\end{equation}

It may be proven through induction that for all $k \ge 2$, $H_{L+k}-G_{L+k}\geq 2\left( H_{L+k-1}-G_{L+k-1} \right)$ (for more details, see Appendix B of the full paper). Applying it to equations \eqref{1first} and \eqref{2first}, we see that \\ $B_{H,\; L+m}-B_{H,\; L+m-1}\leq B_{G,\; L+m}-B_{G,\; L+m-1}$. Summing this inequality to both sides of inequality \eqref{Bequ1}, we arrive at $B_{H,L+m}\leq B_{G,L+m}$, as desired.
\end{proof}


Now, we investigate the behavior when we decrease the last coefficient for any complete sequence. In Theorem \ref{decreaseLastCoe}, we find that decreasing the last coefficient for any complete sequence preserves completeness.

\begin{proof}[Proof of Theorem \ref{decreaseLastCoe}.]
 Given that $\{G_n\}$ is complete, suppose for the sake of contradiction that there exists an incomplete $\{H_n\}$. Thus, let $m$ be the least such that \begin{equation} \label{eq:incomplete}
    H_m>1+\sum^{m-1}_{i=1}H_i.
\end{equation} Simultaneously, as $\{G_n\}$ is complete, by Brown's criterion, \begin{equation}
    G_m\leq1+\sum^{m-1}_{i=1}G_i.
\end{equation} 
First, note that for all $n\leq L$, $G_n=H_n$, hence \begin{equation}
    H_m=G_m\leq 1+\sum^{m-1}_{i=1}G_i=1+\sum^{m-1}_{i=1}H_i,
\end{equation}
which contradicts (\ref{eq:incomplete}). Now, suppose $m>L$. But then by substitution of $G$ for $H$ in the first $L$ terms we obtain 
\begin{equation}
    1+\sum^{L}_{i=1}H_i \geq G_m-\sum_{i=L+1}^{m-1}G_i.
\end{equation}
Moreover,
\begin{equation}
    H_m>1+\sum^{m-1}_{i=1}H_i=1+\sum^{L}_{i=1}H_i+\sum_{i=L+1}^{m-1}H_i\geq G_m-\sum_{i=L+1}^{m-1}G_i+\sum_{i=L+1}^{m-1}H_i,
\end{equation}
and thus
\begin{align}\label{eq:contr}
    H_m-\sum_{i=L+1}^{m-1}H_i&> G_m-\sum_{i=L+1}^{m-1}G_i.
\end{align}
We claim that the opposite of (\ref{eq:contr}) is true, arguing by induction on $m$. For $m=L+1$, we obtain $G_{L+1}\geq H_{L+1}$ as $k_L\leq c_L$. Now, assume that \begin{equation}
    G_m-\sum_{i=L+1}^{m-1}G_i \geq H_m-\sum_{i=L+1}^{m-1}H_i
\end{equation} is true for a positive integer $m$. Using the inductive hypothesis, it then follows that
\begin{align}
    G_{m+1}-\sum_{i=L+1}^{m}G_i=G_{m+1}-\sum_{i=L+1}^{m-1}G_i-G_m&\geq G_{m+1}-2G_m+H_m-\sum_{i=L+1}^{m-1}H_i. 
\end{align}
It may be proven through induction that for all $k\in\mathbb{N}$, $H_{L+k+1}-2H_{L+k}\leq G_{L+k+1}-2G_{L+k}$. Note  
\begin{equation}
    G_{m+1}-2G_m+H_m-\sum_{i=L+1}^{m-1}H_i \geq  H_{m+1}-2H_m+H_m-\sum_{i=L+1}^{m-1}H_i =  H_{m+1}-\sum_{i=L+1}^{m}H_i,
\end{equation}
which does contradict (\ref{eq:contr}) for all $m>L$. 
Therefore, for all $m\in\mathbb{N}$, we have contradicted \eqref{eq:incomplete}. Hence, $\{H_n\}$ must be complete as well.
\end{proof}

The result above is crucial in our characterization of \textit{families} of complete sequences in Section \ref{sec:famiilies}; finding one complete sequence allows us to decrease the last coefficient to find more. Next, we prove two lemmas in the proof of Theorem \ref{Adding M Theorem}.


\begin{lem}\label{IncompExtension}
	Let $\{ G_{n} \}$ be the sequence defined by $[c_1,\ldots, c_{L}]$, and let $\{ H_{n} \}$ be the sequence defined by $[c_1,\ldots, c_{L-1}+1,\; c_{L}-1]$. If $\{ G_{n} \}$ is incomplete, then $\{ H_{n} \}$ must be incomplete as well. 
\end{lem}

\begin{proof}
We claim that for all $m$, $B_{H,m}\leq B_{G,m}$. This lemma is proven using similar reasoning as for Lemma \ref{lem:incompAddCoeff}. We proceed by induction. Clearly, $B_{H,k}=B_{G,k}$ for $1\leq k \le L-1$. Further, for $k=L$, we see \[
	B_{G,L}-B_{H,L}=1+\sum_{i=1}^{L-1}G_{i} - G_{L} - \left(1+\sum_{i=1}^{L-1}H_{i} - H_{L}  \right) = H_{L}-G_{L}=1>0
.\] 
Now, let $m \ge 0$ be arbitrary, and suppose 
\begin{equation}\label{Bequ2}
B_{H,\; L+m}\leq B_{G,\; L+m}.
\end{equation}
We wish to show that $B_{H,\; L+m+1}\leq B_{G,\; L+m+1}$. Note that 
\begin{equation}\label{1second}
B_{H,\; L+m+1}-B_{H,\; L+m}=2H_{L+m}-H_{L+m+1},
\end{equation}
and similarly, 
\begin{equation}\label{2second}
B_{G,\; L+m+1}-B_{G,\; L+m}=2G_{L+m}-G_{L+m+1}.
\end{equation}
Note that for all $k \ge 0$, $H_{L+k+1}-G_{L+k+1}\geq 2\left( H_{L+k}-G_{L+k} \right)$. Applying it to \eqref{1second} and \eqref{2second}, we see $B_{H,\; L+m+1}-B_{H,\; L+m}\leq B_{G,\; L+m+1}-B_{G,\; L+m}$. Summing this inequality to both sides of inequality \eqref{Bequ2}, we conclude that $B_{H,L+m+1}\leq B_{G,L+m+1}$, as desired.
\end{proof}

How many times can Lemma~\ref{IncompExtension} be applied? The answer is all the way up to $[c_1,\ldots,c_{L-1}+c_L-1,1]$, as the last coefficient must remain positive to stay a PLRS.

\begin{lem}\label{Last Case Adding M Theorem}
	Let $\{ G_{n} \}$ be the sequence defined by $[c_1,\ldots , c_{L-1},1]$, and let $\{ H_{n} \}$ be the sequence defined by $[c_1,\ldots , c_{L-1}+1]$. If $\{ G_{n} \}$ is incomplete, then $\{ H_{n} \}$ must be incomplete as well. 
\end{lem}
	
\begin{rek}
	 Despite the similarities, Lemma \ref{Last Case Adding M Theorem} is not implied by Lemma \ref{IncompExtension}; both are necessary for the proof of Theorem \ref{Adding M Theorem}. Applying Lemma \ref{IncompExtension} $(c_L-1)$ times proves that if $[c_1,\ldots , c_{L-1},c_L]$ is incomplete, then $[c_1,\ldots , c_{L-1}+c_L-1,1]$ is incomplete; we cannot apply the lemma further while maintaining a positive final coefficient. Hence the case of Lemma $\ref{Last Case Adding M Theorem}$ must be dealt with separately, in order to prove Theorem \ref{Adding M Theorem}.
\end{rek}
	\begin{proof}
	The proof is similar to that of Lemma \ref{IncompExtension}. We aim to show that $B_{H,m}\leq B_{G,m}$ for all $m$. Clearly $B_{H,k}=B_{G,k}$ for $1\leq k \le L$. Further, for $k=L+1$, we see 
	\begin{equation}
	B_{G,L+1}-B_{H,L+1}=\sum_{i=1}^{L}G_{i}-G_{L+1}-\left( 1+\sum_{i=1}^{L-1}H_{L}-H_{L+1} \right) =H_{L+1}-G_{L+1}=c_1>0.
	\end{equation}
Now, let $m\geq 0$ be arbitrary, and suppose 
\begin{equation}\label{last brown gap inequality}
B_{H,L+m}\leq B_{G,L+m}.
\end{equation}
We wish to show that $B_{H,L+m+1}\leq B_{G,L+m+1}$. Note that 
\begin{equation}\label{ultimobrowngap}
B_{H,L+m+1}-B_{H,L+m}=2H_{L+m}-H_{L+m+1},
\end{equation}
and similarly
\begin{equation}\label{ultimobrowngap 2}
B_{G,L+m+1}-B_{G,L+m}=2G_{L+m}-G_{L+m+1}.
\end{equation}
It may be proven through induction that for all $k\geq 0$, $H_{L+k+1}-G_{L+k+1}\geq 2\left( H_{L+k}-G_{L+k} \right) $  (for more details, see Appendix B of the full paper). Applying it to equations \eqref{ultimobrowngap} and \eqref{ultimobrowngap 2}, we see $B_{H,L+m+1}-B_{H,L+m}\leq B_{G,L+m+1}-B_{G,L+m}$. Summing this inequality to both sides of Inequality \eqref{last brown gap inequality}, we conclude that $B_{H,L+m+1}\leq B_{G,L+m+1}$, as desired.
	\end{proof}
	
Using these lemmas, we can now prove Theorem \ref{Adding M Theorem}.

\begin{proof}[Proof of Theorem \ref{Adding M Theorem}.]
	We apply Lemma \ref{IncompExtension} $c_L-1$ times, to conclude that if $[c_1,\ldots , c_{L-1},c_L]$ is incomplete, then $[c_1,\ldots , c_{L-1}+c_L-1,1]$ is incomplete. Finally, applying Lemma \ref{Last Case Adding M Theorem}, we achieve the desired result.
\end{proof}

\section{Families of Sequences}\label{sec:famiilies}
If we recall Theorem \ref{decreaseLastCoe}, it says that given a complete PLRS, decreasing the last coefficient preserves its completeness. This raises a natural question: Given the first $L-1$ coefficients $c_1, c_2, \dots, c_{L-1}$, what is the maximal $N$ such that $[c_1, c_2, \dots, c_{L-1}, N]$ is complete? While we are not able to answer this question in all generality, in this section, we begin exploring it.

\subsection{Using 1's and 0's as Initial Coefficients}

\begin{proof}[Proof of Theorem \ref{thm:1onekzero}.]
    Suppose that $\{H_n\}$ is complete. By the definition of a PLRS, we can generate the first $k+2$ terms of the sequence simply: $H_i = i$ for all $1 \le i \le k+2$. For all $n > k+1$, we can use the recurrence relation 
\begin{equation}\label{eqn:termsThroughK+4}
    H_{n+1} = H_n + NH_{n-k-1}.
\end{equation}
In the case that $n = k +3$, 
\begin{equation}\label{eq:recrelation}
    H_{k+4} = H_{k+3} + NH_2 = H_{k+3} + 2N.
\end{equation}
As $\{H_n\}$ is complete by supposition, by Brown's criterion, 
\begin{equation}
     H_{k+4} \le H_{k+3} + H_{k+2} + \cdots + H_1 + 1.
\end{equation}
By (\ref{eq:recrelation}), we can replace $H_{k+4}$, so
\begin{equation}
    H_{k+3} + 2N \le H_{k+3} + H_{k+2} + \cdots + H_1 + 1,
\end{equation}
and isolating $N$,
\begin{align}
    N &\le \left[H_{k+2} + H_{k+1} + \cdots + H_1 + 1\right]/{2} \nonumber\\
    &= \left[(k+2) + (k+1) + \cdots + 1 + 1 \right]/{2}\nonumber \\
    &= \frac{(k+2)(k+3)}{4} + \frac{1}{2} \nonumber\\
    \intertext{and as $N$ is an integer,}
    &= \left\lfloor\frac{(k+2)(k+3)}{4} + \frac{1}{2}\right\rfloor \nonumber \\
    &= \left\lceil \frac{(k + 2)(k + 3)}{4}. \right\rceil
\end{align}
Hence, $N \le \left\lceil (k + 2)(k + 3)/{4} \right\rceil$.

We now prove that if $N \leq \left\lceil (k + 2)(k + 3)/{4} \right\rceil$, then $\{H_n\}$ is complete. We first show that if $N_{\max} = \left\lceil (k + 2)(k + 3)/{4} \right\rceil$, then $\{H_n\}$ is complete. Taking the recurrence relation $H_{n+1} = H_n + N_{\max} H_{n-k-1}$, and applying Brown's criterion gives
 \begin{align}
     H_{n+1} &=H_n + N_{\max} H_{n-k-1}\nonumber\\
     &\leq H_n +(N_{\max}-2)H_{n-k-1} + H_{n-k-1} + H_{n-k-2} + \dots + H_1 +1. \nonumber
     \intertext{We can prove by induction that  $(N_{\max}-2)H_{n-k-1} \leq H_{n-1} + \cdots + H_{n-k}$, so}
     H_{n+1} &\leq H_n + H_{n-1} +\dots + H_{n-k} + H_{n-k-1} + H_{n-k-2} + \dots + H_1 +1.
\end{align}
Hence, by Brown's criterion, the sequence $\{H_n\}$ is complete for $N_{\max}$. Lastly, by Theorem \ref{decreaseLastCoe}, for all positive $N < \left\lceil (k + 2)(k + 3)/{4} \right\rceil$, the sequence is also complete.
\end{proof}

Once we have established a result such as Theorem \ref{thm:1onekzero}, it is often possible to allow small additional adjustments to the coefficients while maintaining completeness. In the following corollary, we show that for $L \geq 6$, if we switch one of the coefficients from $0$ to $1$ except for the final zero, then the bound on $N$ to maintain completeness is at least as large.
\begin{cor}
    For $L \geq 6$, given that $[1,0,\dots,0,N]$ is complete, with $N = \left\lceil L(L+1)/{4}\right\rceil$, then $[1,c_2,\dots,c_{L-2},0,N]$ is complete where $c_i=1$ for one $i \in \{2,\dots, L-2\}$, and the rest are $0$.
\end{cor}
\begin{proof}
We begin with the recurrence relation for fixed a $i \in \{2,\dots, L-2\}$,
\begin{align}
    H_{n+1} &= H_n + H_{n-i+1} + NH_{n-L+1}.\\
    \intertext{Applying Brown's criterion on the term $H_{n-L+1}$ gives}
    H_{n+1} &\leq H_n + H_{n-i+1} + (N-2)H_{n-L+1} + H_{n-L+1} + H_{n-L} + \dots + H_1 +1.\\
    \intertext{We can prove by induction that $H_{n-i+1} + (N-2)H_{n-L+1} \leq H_{n-1} + \cdots + H_{n-L+2}$, so}
    &\leq H_n + H_{n-1} + \dots + H_{n-L+2} + H_{n-L+1} + H_{n-L} + \dots + H_1 +1.
\end{align}
Hence, by Brown's criterion, the sequence is complete for all $L \geq 6$.
\end{proof}

\begin{proof}[Proof of Theorem \ref{2onekzero}.]
    Suppose that $\{H_n\}$ is complete. Using the definition of a PLRS, the first $k + 3$ terms of the sequence can be generated in the same way: $H_i = F_{i + 1} - 1$ for all $1 \leq i \leq k + 3$, where $F_n$ is the Fibonacci sequence. Proceeding in a manner similar to the proof of Theorem \ref{thm:1onekzero}, we see that
	\begin{align}
	    H_{k + 4} &= H_{k + 3} + H_{k + 2} + N H_1 = F_{k+5} + N - 2, \nonumber \\
		H_{k+5} &= H_{k+4} + H_{k+3} + NH_2 = F_{k+6} + 3N - 3, \nonumber \\
		H_{k+6} &= H_{k+5} + H_{k+4} + NH_3 = F_{k+7} + 8N - 5.
	\end{align}
	By applying Brown's criterion, 
	\begin{align}
		H_{k + 6} &\leq H_{k + 5} + H_{k + 4} + \cdots + H_1 + 1 \nnend
		&= F_{k+6} + 3N - 3 + F_{k+5} + N - 2 + \sum_{i=1}^{k+3}H_i + 1 \nnend
		&= F_{k+7} + 4N - 5 + \sum_{i=1}^{k+3}(F_{i+1} - 1) + 1.
	\end{align}
	Next,
    \begin{align}
		F_{k + 7} + 8N - 5 &\leq F_{k+7} + 4N - 5 + \sum_{i=1}^{k+3}(F_{i+1} - 1) + F_1, \nnend
		\intertext{which implies}
		4N &\leq \sum_{i=1}^{k+3}(F_{i+1} - 1) + F_1 = \sum_{i = 1}^{k+4}F_i + (k+3)  = F_{k+6} + (k+5).
	\end{align}
	Thus
	\begin{align}
		N &\leq \frac{F_{k + 6} - k-5}{4},\nonumber
	\intertext{and since $N$ is an integer,}
		N &\leq \left\lfloor\frac{F_{k + 6} - k-5}{4}\right\rfloor.
	\end{align}
	Next, we show that if $N = \left\lfloor (F_{k + 6} - k-5)/{4}\right\rfloor$, then $\{H_n\}$ is complete. The initial conditions can be found easily, and for the later terms we have
	\begin{align}
		H_{n + 1} &= H_n + H_{n - 1} + NH_{n - k - 2}\nnend
		&\leq H_n +(N - 2) H_{n - k - 2} + H_{n - k - 2} + H_{n - k - 3} + \cdots + H_1 + 1.\nnend
	\intertext{We can show by induction on $n$ that $(N-2)H_{n-k-2} \le H_{n-1} + \dots + H_{n-k-1}$ for all $n \ge k+3$ and obtain}
		H_{n+1} &\leq H_n + H_{n - 1} + H_{n - 2} + \cdots + H_{n - k - 1} + H_{n - k - 2} + H_{n - k - 3} + \cdots + H_1 + 1.
	\end{align}
	Hence, by Brown's criterion, this sequence is complete. Lastly, by Theorem \ref{decreaseLastCoe}, for all positive $N < \left\lfloor (F_{k + 6} - k-5)/{4}\right\rfloor$, the sequence is also complete.
\end{proof}

\begin{figure}
    \centering
    \includegraphics[scale=0.9]{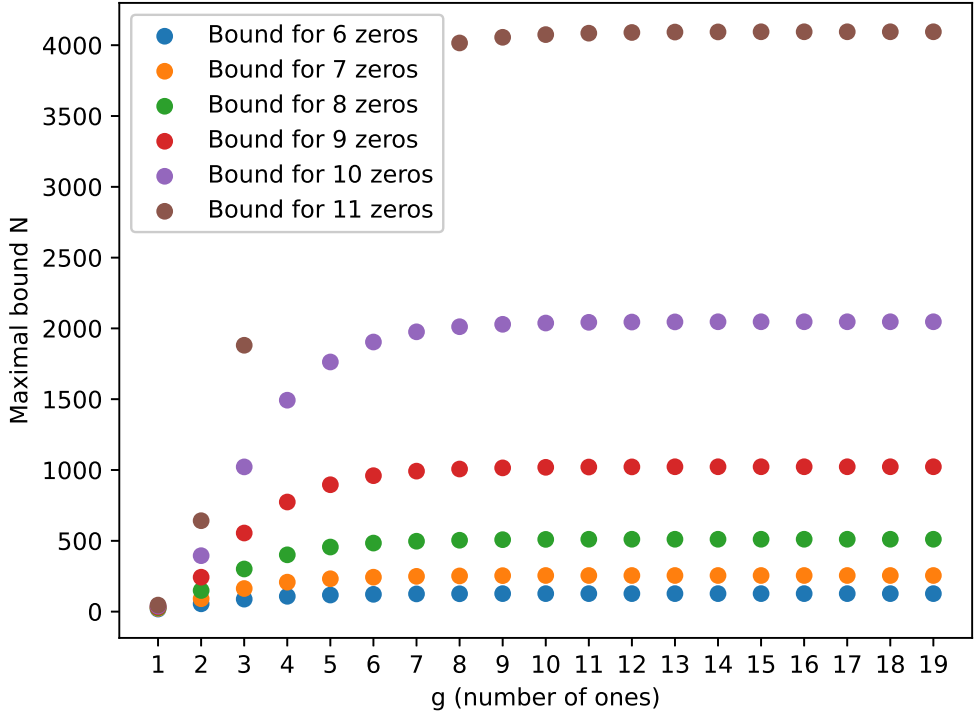}
    \caption{$[\protect\underbrace{1, \dots, 1}_{g}, \protect\underbrace{0, \dots, 0}_{k}, N]$ with $k$ and $g$ varying, where each color represents a fixed $k$.}
    \label{fig:FamiliesOfOne}
\end{figure}
We want to find a more general result for $[\underbrace{1, \dots, 1}_{g}, \underbrace{0, \dots, 0}_{k}, N]$, as seen in Figure \ref{fig:FamiliesOfOne}. Interestingly, we see that as we keep $k$ fixed and increase $g$, the bound increases, and then stays constant from some value of $g$ onward. This observation motivates the following conjecture.
\begin{conj}\label{addFrontOnes}
    If $[\underbrace{1, \dots, 1}_{g}, \underbrace{0, \dots, 0}_{k}, N]$ is complete, then so is $[\underbrace{1, \dots, 1}_{g+1}, \underbrace{0, \dots, 0}_{k}, N]$.
\end{conj}
We have made some progress towards this conjecture; in Theorem \ref{thm:gbon}, we showed the precise bound for $N$ when $g \ge k$.

\subsection{Finitary Criteria for Completeness}

Brown's criterion gives an excellent way to tell whether a sequence is complete, and clearly many useful results on complete PLRS's can be derived from it using induction. However, given the extra recursive structure inherent in PLRS's, it is natural to think that the completeness of these sequences is controlled by the initial conditions, which encode the recurrence coefficients of the sequence. This is particularly useful for deciding whether

It is easy to show that for a given length $L$, there is a bound on the largest term a PLRS $\{H_n\}$ generated by $[c_1, \ldots, c_L]$ can fail Brown's criterion: if $c_i > 2^i$ for any $1 \leq i \leq L$, then $\{H_n\}$ fails Brown's criterion at or before term $i$, and there are a finite number of sequence with coefficients satisfying $c_i \leq 2^i$. So, there is a sequence that fails latest; this shows that not only is there a bound, but that the bound is achieved.

In Lemma~\ref{lem:failAt2L-1} we are able to show that this bound is at least $2L - 1$, since $2L - 1$ is achieved by $[1, \ldots, 1, 0, 4]$. Moreover, no incomplete sequence has been found to fail for the first time after term $2L - 1$, and our conjecture is that the bound is exactly $2L - 1$:

\begin{conj}[The $2L - 1$ Conjecture]\label{2Lcrit}
    The PLRS $\{H_n\}$ defined by $[c_1, \dots, c_L]$ is complete if $B_{H, n} \ge 0$ for all $n \le 2L - 1$, i.e., Brown's criterion holds for the first $2L-1$ terms.
\end{conj}

Proving that $[1, \ldots, 1, 0, 4]$ fails at term $2L - 1$ and not before is a matter of computing the terms, since we know exactly what the sequence is.

\begin{lem}\label{lem:failAt2L-1}
   $[1,\dots,1,0,4]$, with $k \geq 1$ ones, is always incomplete. Moreover, it first fails Brown's criterion on the $(2k+3)$rd term (equivalently, the $(2L - 1)$th term, where $L$ is the number of recurrence coefficients).
\end{lem}

\begin{proof}
We can show that $\{H_n\}$ fails Brown's criterion at term $2k + 3$ by explicitly computing the first terms of the sequence. The $(2k + 3)$rd term is
\begin{equation}
    H_{2k+3} = H_{2k+2} + \dots + H_{k+3} + 4H_{k+1};
\end{equation}
for $1 \le j \le k+1$, we have $H_j=2^{j-1}$, and additionally, $H_{k+2} = 2^{k+1}-1$, so
\begin{align}
    2H_{k+1} = 2^{k+1}>2^{k+1}-1 = H_{k+2};
\end{align}
and finally, $H_{k+1} = H_k +\dots + H_1 +1$. Putting everything together,
\begin{align}
    H_{2k+3} &= H_{2k+2} + \dots + H_{k+3} + 4H_{k+1}\nonumber\\
    &= H_{2k+2} + \dots + H_{k+3} + 3H_{k+1} + H_k + \dots + H_1 +1 \nonumber\\
    &> H_{2k+2} + \dots + H_{k+3} + H_{k+2} + H_{k+1} + H_k + \dots + H_1 +1.
\end{align}
Hence, $[1,\dots,1,0,4]$ is incomplete and in particular, Brown's criterion is failed by the $(2k + 3)$rd term.

Conversely, through a similar computation, we can show Brown's criterion holds for the first $2k + 2$ terms. For $1 \leq j \leq k + 1$, we have $H_j = 2^{j - 1}$, which are the first terms of the complete sequence $\{2^n\}$. On the other hand, when $k + 2 \leq j \leq 2k + 2$, we have $1 \leq j - k - 1 \leq k + 1$, so
\begin{align}
    H_{j + 1} &= H_j + \cdots + H_{j - k + 1} + 4H_{j - k - 1}\nonumber\\
    &= H_j + \cdots + H_{j - k + 1} + 2H_{j - k - 1} + H_{j - k - 1} + (H_{j - k - 2} + \cdots + H_1 + 1)\nonumber\\
    &= H_j + \cdots + H_{j - k + 1} + H_{j - k} + H_{j - k - 1} + H_{j - k - 2} + \cdots + H_1 + 1
\end{align}
as $2H_{j - k - 1} = 2^{j - k + 1} = H_{j - k}$. So Brown's criterion fails for the first time at term $2k + 3$.
%
%
\end{proof}

We can reframe this entire discussion as a question of when the $n$th Brown's gap $B_{H, n}$ falls below 0 for the first time. Our conjecture is then that if $\{H_n\}$ is an incomplete PLRS generated by $[c_1, \ldots, c_L]$, then $B_{H, n} < 0$ for some $n < 2L - 1$. Equivalently, we conjecture that if $B_{H, n} \geq 0$ for all $1 \leq n \leq 2L -1$, then $\{H_n\}$ is complete. This remains a conjecture, but by strengthening the requirement on $B_{H, n}$ for some terms, a similar result can be proven through another computation of terms:

\begin{thm}\label{weak2Lcrit}
    The PLRS $\{H_n\}$ generated by $[c_1, c_2, \dots, c_L]$ is complete if
    \begin{equation}
    \begin{cases}
        B_{H, n} \ge 0 & \text{for } n < L\\
        B_{H, n} > 0 & \text{for } L \le n \le 2L-1.
    \end{cases}
    \end{equation}
\end{thm}
\begin{proof}
    For $L = 1$, an incomplete sequence $[c]$ fails at the second term if and only if $c > 2$. So, we may assume $L \geq 2$. If $c_1 \geq 2$, then the sequence is incomplete as $H_2 = 3$ and $2$ has no representation as a sum of term $H_i$. So we may assume $c_1 = 1$. We show by induction on $n$ that $B_{H, n} > 0$ when $n \geq L$. 
    Suppose $B_{H, n} > 0$ for $L \le n \le m$ (with $m \ge 2L-1$). Then 
\begin{align}
    B_{H, m+1} &= 1 + \sum_{i=1}^mH_i - H_{m+1}\nonumber \\
    &= 1 + \sum_{i=1}^L H_i + \sum_{i=L+1}^m \left(H_{i-1} + \sum_{j=2}^L c_jH_{i-j} \right) - \left(H_m + \sum_{j=2}^L c_jH_{m+1-j}\right)\nonumber\\
    &= \left(1 + \sum_{i=1}^{m-1}H_i - H_m + H_L\right) + \sum_{j=2}^L c_j\left(\sum_{i=L+1}^m H_{i-j} - H_{m +1-j} \right)\nonumber \\
    &= (B_{H, m} + H_L) + \sum_{j=2}^L c_j\left(B_{H, m+1-j} - 1 - \sum_{i = j+1}^L H_{i-j} \right) \nonumber \\ 
    &= B_{H, m} + \sum_{j=2}^L c_j(B_{H, m+1-j} - 1) + H_L - \sum_{i = 3}^L\sum_{j=2}^{i-1} c_jH_{i-j}\nonumber \\ 
    &= B_{H, m} + \sum_{j=2}^L c_j(B_{H, m+1-j} - 1) + H_L - \sum_{i=3}^L(H_i - H_{i-1} - 1)\nonumber \\
    &= B_{H, m} + \sum_{j=2}^L c_j(B_{H, m+1-j} - 1) + L.
\end{align}
The last line is positive since $B_{H, m+1-j} - 1 \ge 0$ and $B_{H, m}, L > 0$. This completes the induction; hence $\{H_n\}$ is complete.
\end{proof}

This result is essentially as good as Conjecture~\ref{2Lcrit} as a sufficient criterion for a sequence to be complete; however, the two results differ in strength because Conjecture~\ref{2Lcrit} gives a necessary and sufficent condition. The condition that a PLRS not fail Brown's criterion in the first $2L - 1$ terms is certainly necessary for the PLRS to be complete, as failure of Brown's criterion shows that the sequence is incomplete. The conjecture is then that this is also sufficient for the sequence to be complete, and Theorem~\ref{weak2Lcrit} proves a weaker sufficient condition.


\end{document}